\documentclass[BCOR=12mm, DIV=default,twoside, pagesize, draft]{scrartcl}

\addtokomafont{caption}{\small}
\setkomafont{captionlabel}{\sffamily\bfseries}

\usepackage{tikz}    
\usetikzlibrary{arrows,automata}   
 
\usepackage{amsmath, amssymb, amsthm} 
\usepackage[utf8x]{inputenc}
\usepackage{libertine} 
\usepackage[T1]{fontenc}
\usepackage[english]{babel}
\usepackage{url}  
\usepackage{faktor,csquotes,color}
\usepackage{pgfplots} 
\usepackage{enumitem}

		\usepackage{float} 
		\usetikzlibrary{arrows}
		\usetikzlibrary{datavisualization.formats.functions}
		\pgfplotsset{compat=1.12}
		\usepgfplotslibrary{fillbetween}
		\usetikzlibrary{patterns}
		\usepackage{accents}

\usepackage{changes}

\definechangesauthor[name={Kristoffer}, color=blue]{kri}
\newcommand{\stkout}[1]{\ifmmode\text{\sout{\ensuremath{#1}}}\else\sout{#1}\fi}
\setdeletedmarkup{\stkout{#1}}

\usepackage{units,varioref}

\usepackage[shortcuts]{extdash}

\usepackage{scrlayer-scrpage}
\KOMAoptions{draft=false}

\usepackage{mathtools}
\mathtoolsset{showonlyrefs,showmanualtags}
\numberwithin{equation}{section}

\allowdisplaybreaks 
\clearscrheadfoot
\automark[section]{section}
\ohead{\headmark}
\ofoot[\pagemark]{\pagemark}

\newtheoremstyle{break}{\topsep}{\topsep}{\itshape}{}{\bfseries}{.}{\newline}{}
\newtheoremstyle{exampl}{\topsep}{\topsep}{\upshape}{}{\bfseries}{.}{\newline}{}

\theoremstyle{plain}
\newtheorem{thm}{Theorem}[section]
\newtheorem{lem}[thm]{Lemma}  
\newtheorem{prop}[thm]{Proposition}  
 
\newtheorem{ass}[thm]{Assumption} 

\theoremstyle{break}

\theoremstyle{definition}
\newtheorem{defi}[thm]{Definition}

\theoremstyle{exampl}

\theoremstyle{remark}
\newtheorem{rem}[thm]{Remark}

\def\R{\mathbb{R}}

\definecolor{mygray}{gray}{.5}

\title{De Finetti's control problem with competition}
\author{Erik Ekstr\"om\footnote{Department of Mathematics, Uppsala University, Sweden. 
E-mail address: ekstrom@math.uu.se.} \and Kristoffer Lindensj\"o\footnote{Department of Mathematics, Stockholm University, Sweden. E-mail address: kristoffer.lindensjo@math.su.se.}}

\date{\today}

\begin{document}
\maketitle

\begin{abstract} 
We investigate the effects of competition in a problem of resource extraction from a common source with diffusive dynamics. In the symmetric version with identical extraction rates we prove the existence of a Nash equilibrium where the strategies are of threshold type, and we characterize the equilibrium threshold. 
Moreover, we show that increased competition leads to lower extraction thresholds and smaller equilibrium values.
%
For the asymmetric version, where each agent has an individual extraction rate, we provide the existence of an equilibrium in threshold strategies, and we show that the corresponding thresholds are ordered in the same way as the extraction rates.
\end{abstract}

\noindent \textbf{Keywords:} 
De Finetti's control problem,  
Optimal dividend problem, 
Optimal resource extraction, 
Nash equilibrium, 
Stochastic game.
\vspace{1mm}
 
\noindent \textbf{AMS MSC2010:}  
93E20; 60J70; 91A10; 91A25; 91B51.

\section{Introduction} \label{sec:intro}
 
With roots in the classical Faustmann problem of forest rotation (for modern references, see \cite{alvarez2004stochastic,alvarez2007optimal,willassen1998stochastic}), the research on resource 
extraction problems in random environments is vast. 
In this literature, a trade-off between profitable resource extraction and sustainability is an underlying theme:
a business wants to extract resources at a reasonable rate to gain profit, while a too aggressive extraction strategy may lead to extinction. Many articles focus on the single-player problem, where a single agent extracts resources from a stochastically fluctuating population (see \cite{
elsanosi2000some,hening2019asymptotic,
lande1994optimal,lungu2001optimal}).  
The problem is closely connected with the dividend problem in finance, also known as the De Finetti problem (see \cite{AT,ZS}). However, there are also application specific differences in the model set-ups, such as different dynamics used for the underlying process, different boundary conditions at extinction/default, differences regarding whether exerting control is costly or not, and whether the extraction rate is controlled directly or indirectly via an effort control.
 
In many applications, however, several agents are present and thus strategic considerations need to be taken into account. 
In \cite{jorgensen1996stochastic,wang2010stochastic}, the specific case of competitive resource extraction from a common pool of fish is considered in particular diffusion models, and with specific choices of market prices and extraction costs.  In such a framework, \cite{jorgensen1996stochastic} obtains an explicit feedback Nash equilibrium in which extraction rates are proportional to the current stock of biomass.
In an extended model allowing for a two-species fish population with interaction, \cite{wang2010stochastic}
shows that the corresponding results also hold in higher dimensions. Related literature includes \cite{haurie1994monitoring,kaitala1993equilibria} where
cooperative equilibria are discussed, and \cite{sandal2004dynamic} in which a myopic setting is used.
On the financial side, 
the reference that is closest in spirit that we have found is
\cite{ZCJL}, in which a non-zero-sum game between insurance companies is studied where the control of each company is a dividend flow as in the De Finetti problem.

The objective of the present paper is to study the effect of strategic considerations for resource extraction 
in a random environment. To isolate the effects of the introduction of strategic elements, we work
in a flexible diffusion model with absorption, but stripped from most application specific features (such as costly controls
and effort controls), and with the objective functions being simply the value of discounted accumulated resource extraction until extinction. In this setting, we show that if the agents have identical maximal extraction rates, then a symmetric feedback Nash equilibrium in threshold strategies (i.e. strategies where maximal extraction is employed above a certain threshold and no extraction is employed below it) is obtained. Moreover, we demonstrate the natural result that increased competition lowers the equilibrium threshold and the individual equilibrium value. Furthermore, the total equilibrium value decreases in the number of agents if the individual extraction rates are decreased so that the total extraction rate is kept constant, and extinction is 
accelerated. Added competition is thus inhibiting both for total profitability and for sustainability. 
%
Finally, an asymmetric version is studied, in which each agent is equipped with an individual maximal extraction rate. Again, an equilibrium in threshold strategies is obtained, and the corresponding thresholds are shown to be ordered in the same way as the extraction rates.


%

\section{Mathematical setup}

On a family of complete filtered probability spaces $\{(\Omega,\mathcal{F},(\mathcal{F}_t)_{t\geq0},\mathbb{P}_x), x\geq 0\}$ satisfying the usual conditions we consider a one-dimensional process $X=(X_t)_{t\geq 0}$ --- corresponding to the total value of a scarce resource --- given under $\mathbb{P}_x$ by 
 \begin{equation} 
dX_t =  \left(\mu(X_t)-\sum_{i=1}^n{\lambda_{it}}\right)dt + \sigma(X_t)\, dB_t, \enskip X_0=x,\label{state-process}
\end{equation}
where $\mu$ and $\sigma$ are given functions (satisfying certain conditions, see Assumption \ref{coeff-assum} below), 
$B=(B_t)_{t\geq 0}$ is a one-dimensional Wiener process, and $\lambda_i =(\lambda_{it})_{t\geq 0},i=1,...,n$ are non-negative controls that satisfy $\lambda_{it}\leq K$ at all times, where $K>0$ is a given maximal 
extraction rate. In view of the given Markovian structure, we focus on Markovian controls of the type $\lambda_{it}=\lambda_i(X_t)$, where $\lambda_i:[0,\infty)\to[0,K]$ is measurable. We refer to such a Markovian control as
an {\em admissible (resource extraction) strategy}, and an $n$-tuple $\boldsymbol{\lambda}=(\lambda_1,...,\lambda_n)$ of admissible strategies is said to be an admissible strategy {\em profile}.
We note that the conditions on the coefficients $\mu$ and $\sigma$ specified in Assumption~\ref{coeff-assum} below guarantee that \eqref{state-process} has a unique strong solution for any given admissible strategy profile, cf. \cite[Proposition 5.5.17]{KS}.
For a given admissible strategy profile $\boldsymbol{\lambda}$, the reward of each agent $i$ is given by
\begin{align}  \label{payoff-func}
\enskip J_i(x;\boldsymbol{\lambda}) := \mathbb{E}_x\left(\int_0^{\tau} e^{-rt}{\lambda_{it}}\,dt\right), 
\end{align}
where $r>0$ is a constant discount rate, ${\tau} :=\inf\{t\geq 0: X_t\leq0\}$ is the extinction time and $\lambda_{it}:=\lambda_i(X_t)$.

In this setting we allow each agent $i$ to select the strategy $\lambda_i$. Naturally, each agent $i$ seeks to make this choice in order to maximize the reward \eqref{payoff-func}, and we define a Nash equilibrium accordingly. 
Note that while our main emphasis is on the strictly competitive case $n\geq 2$, 
our results also hold for the one-player game with $n=1$. In that case, the strategic element is eliminated, and the set-up
reduces to a classical problem of stochastic control.

\begin{defi} [Nash equilibrium] \label{equilibrium-def} For a fixed initial value $x \geq 0$, 
an admissible strategy profile $\boldsymbol{\hat \lambda}=(\hat \lambda_1,...,\hat\lambda_n)$ is a Nash equilibrium if, for each $i=1,...,n$, 
\begin{align}  
J_i(x; \boldsymbol{\hat \lambda}) \geq J_i(x;({\lambda_{i}},\boldsymbol{\hat \lambda}_{-i})) \enskip
\label{NE-eq}
\end{align}
for any admissible strategy profile $({\lambda_{i}},\boldsymbol{\hat \lambda}_{-i})$, which is the strategy profile obtained when replacing $\hat \lambda_i$ with $\lambda_i$ in $\boldsymbol{\hat \lambda}$. The corresponding equilibrium values are $J_i(x; \boldsymbol{\hat \lambda}),i=1,...,n$.
\end{defi}

\textbf{Assumptions and preliminaries}: Throughout the paper we rely on the following standing assumption
(the only exception is Lemma~\ref{lemma1}, where we only assume \ref{coeff-assum:1}-\ref{coeff-assum:3}).

\begin{ass} \label{coeff-assum} \quad 
\begin{enumerate}[label=(A.\arabic*)] 

\item  \label{coeff-assum:1} 

The functions $\sigma,\mu\in {\cal C}^1([0,\infty))$.
\item \label{coeff-assum:2} 
$\sigma^2(x)>0$ and $\vert\sigma(x)\vert +\vert\mu(x)\vert\leq C(1+x)$ for some $C>0$, for all $x\geq0$. 
\item \label{coeff-assum:3} 
The function $\mu$ satisfies $\mu'(x) < r$ for all $x\geq0$.
Moreover, there exists a constant $c>0$ such that $\mu(x)\leq rx-c^{-1}$ 
for $x\geq c$.
\item \label{coeff-assum:4} 
$\mu(0) > 0$. 
\end{enumerate}
\end{ass}

For any fixed constant $A\geq 0$ we define the downwards shifted drift function 
\begin{align}
\mu_A(x)= \mu(x)- A,
\end{align}
and we denote by $\psi_A$ ($\varphi_A$) a positive and increasing (decreasing) solution to 
\begin{align}
 \frac{1}{2}\sigma^2(x)f''(x)+\mu_A(x)f'(x) -rf(x)=0 \label{L-diff-0}
\end{align} 
with $\psi_A(0)=\varphi_A(\infty)=0$. Then, $\psi_A$ and $\varphi_A$ are unique up to multiplication with positive constants, 
cf. \cite[p. 18-19]{borodin2012handbook}, and $\mathcal C^3([0,\infty))$. 
We use the normalization given by $\psi_A'(0)=1$ and $\varphi_A(0)=1$. 
If $A=0$ so that $\mu_A = \mu$, then we for notational convenience write $\psi$ and $\varphi$ instead of $\psi_0$ and $\varphi_0$. 

Most parts of the following lemma are well-known, compare
\cite[Lemma 2.2]{paulsen2007optimal}, \cite[Proposition 2.5]{bai2012optimal} 
and \cite[Lemma 4.1]{shreve1984optimal}. For the convenience of the reader, however, 
we provide a full proof.

\begin{lem} \label{lemma1}  Under \ref{coeff-assum:1}--\ref{coeff-assum:3}:
\begin{enumerate} [label=(\roman*)]
\item
$\psi'(x)>0$ for $x\geq 0$ 
and $\psi$ is (strictly) concave-convex on $[0,\infty)$ with a unique inflection point $b^* \in[0,\infty)$. 
In fact, if  $\mu(0)>0$ then $b^*>0$, and $\psi''(b^*)=0$, $\psi''(x)<0$ for $x<b^*$ and $\psi''(x)>0$ for $x>b^*$; 
and otherwise $b^*=0$ and $\psi''(x)>0$ for $x> 0$. 


\item $\varphi_A'(x)<0$ and $\varphi_A''(x)>0$ for $x\geq 0$, for all $A \geq 0$.
\end{enumerate} 
\end{lem}

\begin{proof} 
First we show that $\psi'>0$ everywhere. To see this, note that $\psi'(0)=1$, and if $\psi'(x)=0$ at some $x>0$, then $\psi'$ has a local minimum at $x$ so that $\psi''(x)=0$. Consequently, 
\[0=\frac{\sigma^2}{2}\psi''=-\mu\psi'+r\psi=r\psi>0\]
at $x$, which is a contradiction, so $\psi'>0$ everywhere.

Next assume that $\psi''(x)\leq 0$ at some point $x>0$, and assume that $y$ is a point in $(0,x)$ such that $\psi''(y)=0$ and $\psi''\leq 0 $ on $(y,x)$. Then $\psi'''(y)\leq 0$, but we also have 
\[\frac{\sigma^2}{2}\psi'''=-(\mu+ \sigma\sigma')\psi''+(r-\mu')\psi'=(r-\mu')\psi'>0\]
at $y$ by \ref{coeff-assum:3}.
The existence of a unique inflection point $b^*$ (possibly taking degenerate values 0 or $\infty$)
thus follows. 

To rule out $b^*=\infty$, note that in that case $\psi$ is concave on $[0,\infty)$, 
so 
\[0\geq \frac{\sigma^2}{2}\psi'' =-\mu\psi'+r\psi\geq -\mu^+\psi'+r\psi
\geq (rx-\mu^+)\frac{\psi}{x}>0\]
for $x$ large, where $\mu^+=\max\{\mu,0\}$. Consequently, $b^*<\infty$.
Moreover, $\frac{\sigma^2}{2}\psi''(0)=-\mu(0)$, so $b^*>0$ if $\mu(0)>0$
and $b^*=0$ if $\mu(0)<0$. If $\mu(0)=0$, then $\psi''(0)=0$, but then 
$\psi'''(0)=\frac{2(r-\mu'(0))}{\sigma^2(0)}>0$ so that $b^*=0$ also in this case.

For (ii), first note that if $\varphi'_A(x)=0$ at some $x\geq 0$, then 
$\varphi''_A(x)=\frac{2r}{\sigma^2(x)}\varphi_A(x)> 0$ so that $\varphi'_A>0$ in a right neighborhood of $x$, which contradicts the monotonicity of $\varphi_A$. Thus $\varphi'_A<0$ everywhere.

Now, assume that there exists a point $x\geq 0$ with $\varphi_ A''(x)=0$. Then
\[\frac{\sigma^2}{2}\varphi_A'''=-(\mu+ \sigma\sigma')\varphi_A''+(r-\mu')\varphi_A'=(r-\mu')\varphi_A'<0\]
at $x$, so then $\varphi_A$ is concave on $[x,\infty)$. This is, however, impossible, since $\varphi_A$ is strictly decreasing and bounded. It follows that $\varphi_A''>0$ everywhere.
\end{proof}  

\begin{rem} \label{remark-optimal-div}
It is well-known that the inflection point $b^*$ is the level at which $X$ should be reflected in the single-agent version ($n=1$) of the game described above when allowing for singular extraction rates. 
Specifically, consider the problem an agent would face if she were alone ($n=1$) and could instead of $\int_0^t\lambda_{1t}ds$, cf. \eqref{state-process}, select any non-decreasing adapted RCLL process $\Lambda$ (satisfying admissibility in the sense that $\Lambda_{0-}=0$ and $X_{\tau}\geq 0$) in order to maximize the reward
\[\mathbb{E}_x\left(\int_0^{\tau} e^{-rt} d\Lambda_{t}\right), \]
(cf. \eqref{payoff-func}); see \cite{shreve1984optimal} for the original formulation. In the present setting, this problem has the value function 
\begin{equation}\label{singular1}
U(x)=	
\begin{cases}  
			C^*\psi(x), 					&  0\leq x\leq b^*,\\
			x- b^* + C^*\psi(b^*), 					&  x \geq b^*,
			\end{cases}
\end{equation}
where $C^*= (\psi'(b^*))^{-1}$ so that $U'(b^*)=1$ and $U''(b^*)=0$, with $b^*>0$ since $\mu(0)>0$, and the optimal policy is to reflect the state process $X$ at a barrier consisting of the inflection point $b^*$, with an immediate dividend of $x-b^*$ in case $x > b^*$; see \cite[Theorem 4.3]{shreve1984optimal} and \cite[Proposition 2.6]{christensen2019moment}.
\end{rem}

\section{A threshold Nash equilibrium} \label{identical-NE}
The main objective of the present paper is to find and study Nash equilibria (Definition \ref{equilibrium-def}) of threshold type, i.e. Nash equilibria where each extraction rate $\lambda_i$ has the form
\begin{align} 
\lambda_{it}=KI_{\{ X_t\geq b_i\}}
\enskip \mbox{for some constant $b_i\geq0$ and all $t\geq0$}.
\end{align}
With a slight abuse of notation we write a strategy profile comprising only threshold strategies as $\boldsymbol{\lambda}=(b_1,...,b_n)$.


\subsection{Deriving a candidate Nash equilibrium} \label{sec:problem:ident}

In this section we derive a candidate Nash equilibrium of threshold type. We remark that this section is mainly of motivational value and that a verification result is reported in Section \ref{sec:verthm1}.

Since the agents are identical it is natural to search for a symmetric Nash equilibrium, i.e. an equilibrium of the kind $\boldsymbol{\hat \lambda} = (\hat \lambda,...,\hat \lambda)$
for some admissible strategy $\hat \lambda =(\hat\lambda_t)_{t\geq 0}$. Clearly, if $\boldsymbol{\hat \lambda}$ is such a symmetric equilibrium, then the corresponding equilibrium values are all identical,
and we will denote this shared equilibrium value function by $V= \{V(x),x\geq 0\}$.

Now, if agent $n$ (who is singled out without loss of generality) deviates from such an equilibrium by using an alternative strategy $\lambda$, then the corresponding state process is given by 
\begin{align}\label{stateP-in-heur1}
dX_t =  \left(\mu(X_t)- (n-1)\hat\lambda(X_t)-\lambda(X_t) \right)dt + \sigma(X_t) dB_t.
\end{align}
For $\boldsymbol{\hat\lambda}$ to be a Nash equilibrium it should be the case that setting $\lambda=\hat\lambda$ in \eqref{stateP-in-heur1} maximizes the corresponding reward for agent $n$ and that this strategy gives the equilibrium value $V$. Hence, from standard martingale arguments it follows that $V$ should satisfy
\begin{align}
\label{ODEyaayay0}
\frac{1}{2}\sigma^2(x)V''(x)+\left(\mu(x)-(n-1)\hat\lambda(x)-\lambda\right)V'(x) -rV(x) + \lambda & \leq 0
\end{align}
for all $\lambda\in [0,K]$ and $x>0$, and that equality in \eqref{ODEyaayay0} should be obtained when $\lambda = \hat \lambda$, i.e. 
\begin{align} 
\label{ODEyaayay1}\frac{1}{2}\sigma^2(x)V''(x)+\left(\mu(x)-n\hat\lambda(x)\right)V'(x) -rV(x) + \hat\lambda(x)= 0.
\end{align}
Subtracting Equation \eqref{ODEyaayay1} from Equation \eqref{ODEyaayay0} yields 
\begin{align} 
\left(V'(x)-1\right)(\hat\lambda(x) - \lambda)\leq 0
\end{align}
for all $\lambda\in[0,K]$.
Hence one finds that
\begin{equation} \label{asdasdas2131}
\hat\lambda(x)=	
\begin{cases}
			0, 		&x \in  \{x: V'(x)>1\}\\
			K,		&x \in \{x: V'(x)<1\},
		\end{cases}
\end{equation} 
i.e.  
\begin{itemize}
\item[(i)] each agent extracts resources according to the maximal extraction rate $K$ when $V'(x)<1$, 
\item[(ii)] no resources are extracted when $V'(x)>1$.
\end{itemize}

Since we seek a Nash equilibrium of threshold type we now make the Ansatz
$\hat \lambda(x)= KI_{\{ x\geq \hat b\}}$ 
for some threshold $\hat b \geq 0$. In this case it should hold, cf. \eqref{ODEyaayay1}, that
\begin{equation}\label{ODE1}
\frac{1}{2}\sigma^2(x)V''(x)+\mu(x) V'(x)-rV(x)  = 0, \mbox{ for } 0<x <\hat b 
\end{equation}
and
\begin{equation}
\label{ODE2}
\frac{1}{2}\sigma^2(x)V''(x)+\left(\mu(x)-nK\right)V'(x)-rV(x) + K = 0, \mbox{ for } x > \hat b
\end{equation}
as well as $V'(\hat b)=1$ in case $\hat b>0$. We treat the cases $\hat b=0$ and $\hat b>0$ separately. 

Suppose first that $\hat b=0$. Solving \eqref{ODE2} then yields $V(x)=C_1\psi_{nK}(x)+C_2\varphi_{nK}(x)+\frac{K}{r}$ for constants $C_1$ and $C_2$. Imposing $V(\infty)=K/r$ (corresponding to maximal extraction for an infinite time) and $V(0)=0$ yields $C_1=0$ and $C_2=-K/r$ so that 
\begin{align} \label{V-in-case-hatb-is0}
V(x)=\frac{K}{r}(1-\varphi_{nK}(x)).
\end{align}
Moreover, in view of \eqref{asdasdas2131}, we find that $V'(x)\leq 1$ must hold for all $x\geq 0$. In view of
the convexity of $\varphi_{nK}$, we conclude that the case $\hat b = 0$ requires that 
\begin{equation}\label{cond0}
\varphi_{nK}'(0)\geq -r/K.
\end{equation}

Next, if instead
\begin{equation}\label{cond}
\varphi_{nK}'(0)< -r/K,
\end{equation} 
then we expect that $\hat b>0$ and a corresponding equilibrium value function of the form
\begin{equation}
V(x)=
\begin{cases}	
	D_1\psi(x)+D_2\varphi(x), 					&  0<x <\hat b,\\
	D_3\psi_{nK}(x) + D_4\varphi_{nK}(x)+ \frac{K}{r}, &  x > \hat b,\\
\end{cases}
\end{equation}
for some constants $D_i$, $i=1,...,4$ and $\hat b$ to be determined. Imposing the boundary conditions $V(0)=0$ and $V(\infty)=K/r$ gives $D_2=0$ and $D_3=0$. We also make the Ansatz that $V(x)$ is differentiable at $\hat b$, which determines $D_1$ and $D_4$. We thus obtain
\begin{equation}\label{asdasdff222}
V(x)=	
\begin{cases}
			D_1\psi(x), 					&  0<x\leq \hat b,\\
			D_4\varphi_{nK}(x) + \frac{K}{r},		& x \geq \hat b, 
		\end{cases}
\end{equation}
where
\begin{align}
D_1=\frac{-\frac{K}{r}\varphi_{nK}'(\hat b)}{\varphi_{nK}(\hat b)\psi'(\hat b)-\varphi_{nK}'(\hat b)\psi(\hat b) }\label{C-def}
\end{align}
and 
\begin{equation}\label{D_4}
D_4= 
\frac{-\frac{K}{r}\psi'(\hat b)}{\varphi_{nK}(\hat b)\psi'(\hat b) -\varphi_{nK}'(\hat b)\psi(\hat b)}.
\end{equation}
Finally, in line with the arguments above, we determine $\hat b$ by imposing the condition 
\begin{equation}
\label{smoothfit}
V'(\hat b)=1. 
\end{equation}
Differentiating \eqref{asdasdff222} thus gives the equation
\begin{equation}
\frac{K}{r}\frac{\psi'(\hat b)\varphi_{nK}'(\hat b)}{\varphi_{nK}'(\hat b)\psi(\hat b)-\varphi_{nK}(\hat b)\psi'(\hat b) }=1,
\end{equation}
which can be rewritten as
\begin{equation}
\label{b}
\frac{\psi(\hat b)}{\psi'(\hat b)}-\frac{\varphi_{nK}(\hat b)}{\varphi_{nK}'(\hat b)}=\frac{K}{r}.
\end{equation}

\begin{lem} \label{lem:bhat-is-pos-sol}
Suppose \eqref{cond} holds. Then Equation \eqref{b} has a unique positive solution $\hat b$. Moreover, $0<\hat b \leq b^*$, where $b^*$ is the unique inflection point of $\psi$ (see also Remark \ref{remark-optimal-div}).
\end{lem}
 
\begin{proof}
Define the function $f:[0,\infty)\to\R$ by
\begin{equation}\label{f}
f(b)=\frac{\psi(b)}{\psi'(b)}-\frac{\varphi_{nK}(b)}{\varphi_{nK}'(b)},
\end{equation}
and note that $f(0)=-1/\varphi_{nK}'(0)<K/r$ by \eqref{cond}. Thus any solution to the equation $f(\hat b)=K/r$ satisfies $\hat b>0$.
Differentiation shows that 
\begin{align}\label{fprime}
\frac{\sigma^2(b)}{2}f'(b) &= \frac{\sigma^2(b)}{2}\left(\frac{\varphi_{nK}(b)\varphi_{nK}''(b)}{(\varphi_{nK}'(b))^2}-\frac{\psi(b)\psi''(b)}{(\psi'(b)) ^2}\right)\\
&= \frac{f(b)}{\varphi_{nK}'(b)}(\mu_{nK}(b)\varphi_{nK}'(b)-r\varphi_{nK}(b))+\frac{\psi(b)}{\psi'(b)}(nK-rf(b))\\
&\geq \frac{\psi(b)}{\psi'(b)}(nK-rf(b)),
\end{align}
where we have used the ODEs that $\varphi_{nK}$ and $\psi$ satisfy and, in the last inequality, that $\varphi_{nK}$ is convex so that
\[\mu_{nK}(b)\varphi_{nK}'(b)-r\varphi_{nK}(b)=-\frac{\sigma^2(b)}{2}\varphi_{nK}''(b)\leq 0.\] 
It follows that $f$ is increasing as long as $f\leq nK/r$. Consequently, the equation $f(\hat b)=K/r$ has at most one positive solution. 
Moreover, by similar calculations as in \eqref{fprime},
\begin{align}
\frac{\sigma^2(b)}{2}f'(b) &=\frac{f(b)}{\psi'(b)}(\mu\psi'(b)-r\psi(b))+\frac{\varphi_{nK}(b)}{\varphi_{nK}'(b)}(nK-rf(b)),
\end{align}
so at the inflection point $b^*$ of $\psi$ 
we have 
\[\frac{\sigma^2(b^*)}{2}f'(b^*)=\frac{\varphi_{nK}(b^*)}{\varphi_{nK}'(b^*)}(nK-rf(b^*)).\]
Comparing with \eqref{fprime} we find that
\[\frac{\psi(b^*)}{\psi'(b^*)}(nK-rf(b^*))\leq \frac{\varphi_{nK}(b^*)}{\varphi_{nK}'(b^*)}(nK-rf(b^*)),\]
which implies that $f(b^*)\geq nK/r$. Thus, by continuity there exists a point $\hat b\in(0,b^*]$ such that $f(\hat b)=K/r$, which completes the proof.
\end{proof}

Note that if $\hat b=0$, then $D_4=-K/r$ and the representation of $V$ in \eqref{asdasdff222} yields \eqref{V-in-case-hatb-is0}. Let us summarize the above discussion by specifying the candidate equilibrium threshold $\hat b$ and the corresponding equilibrium value function $V$. If \eqref{cond0} holds, then we define $\hat b=0$. If instead \eqref{cond} holds, then we define $\hat b$ as the unique positive solution of \eqref{b}. In both cases
\begin{align}\label{eqvalbhatnot0}
V(x) = \begin{cases}
			D_1\psi(x), 					&  0<x\leq \hat b,\\
			D_4\varphi_{nK}(x) + \frac{K}{r},		& x \geq \hat b, 
		\end{cases}
\end{align}
with $D_1$ and $D_4$ specified in \eqref{C-def}-\eqref{D_4}. Note that $V \in \mathcal C^1([0,\infty)) \cap \mathcal C^2([0,\infty)\setminus\{\hat b\})$ and that it satisfies \eqref{ODE1}-\eqref{ODE2}. Moreover, if $\hat b>0$ then also \eqref{smoothfit} holds.

\begin{lem}\label{conc}
The function $V$ is concave.
\end{lem}

\begin{proof}
If \eqref{cond0} holds, then 
$D_4=-K/r$ and so $V(x)=\frac{K}{r}(1-\varphi_{nK}(x))$. The concavity of $V$ thus follows from convexity of $\varphi_{nK}$.
If instead \eqref{cond} holds, 
then the claim is verified using \eqref{eqvalbhatnot0}, Lemma \ref{lem:bhat-is-pos-sol}  and Lemma \ref{lemma1}.
\end{proof}


%
 
\subsection{A verification theorem for the threshold Nash equilibrium} \label{sec:verthm1}


\begin{thm}  \label{verifi-thm}  A Nash equilibrium of threshold type exists. In particular, define 
$\hat b=0$ if \eqref{cond0} holds, and let 
$\hat b$ be the unique solution to \eqref{b} if instead \eqref{cond} holds. Then, $\boldsymbol{\hat \lambda}=(\hat b,...,\hat b)$ is a Nash equilibrium and the corresponding equilibrium value functions are identical and given by $V$ defined in \eqref{eqvalbhatnot0}.
\end{thm}

\begin{proof} Suppose an arbitrary agent $i$ deviates from $\boldsymbol{\hat \lambda}=(\hat b,...,\hat b)$. More precisely, consider the strategy profile $(\lambda,\boldsymbol{\hat \lambda}_{-i})$ (cf. Definition \ref{equilibrium-def}), where $\lambda$ 
is an arbitrary admissible strategy. Then $X$ is given by  
\begin{align} 
dX_t =  \left(\mu(X_t)-\lambda(X_t)-(n-1)KI_{\{X_t\geq \hat b\}}\right)dt + \sigma(X_t)\,dB_t.
\end{align}
Using It\^o's formula we obtain
\begin{align}
 e^{-r(\tau \wedge \theta_t)}V(X_{\tau \wedge \theta_t }) 
& =  V(x)  + \int_0^{\tau \wedge \theta_t } e^{-rs} 
\left(\mu(X_s)-\lambda(X_s)-(n-1)KI_{\{X_s\geq \hat b\}} \right)V'(X_s)\,ds\\
& \quad+\int_0^{\tau \wedge \theta_t }e^{-rs} 
\left(\frac{1}{2}\sigma^2(X_s)V''(X_s)I_{\{X_s\neq \hat b\}}-rV(X_s)\right)ds\\
&\quad+ \int_0^{\tau \wedge \theta_t } e^{-rs} \sigma(X_s)V'(X_s)\,dB_s,
\end{align}
where $\theta_t:= \inf\{s\geq 0: X_s \geq t\}\wedge t$ (with the standard convention that the infimum of the empty set is infinite). It is directly verified, cf. \eqref{ODE1}-\eqref{ODE2}, that 
\begin{align}
\frac{1}{2}\sigma^2(x)V''(x)+\left(\mu(x)-nKI_{\{x\geq \hat b\}}\right)V'(x)-rV(x)+KI_{\{x\geq \hat b\}} = 0
\end{align}
for all $x\neq \hat b$,
which implies that 
\begin{align}
 e^{-r{\tau \wedge \theta_t}}V(X_{\tau \wedge \theta_t }) 
& =  V(x)  + \int_0^{\tau \wedge \theta_t } e^{-rs} 
\left(   \left(V'(X_s)-1\right)\left( KI_{\{X_s\geq \hat b\}}-\lambda(X_s) \right)-\lambda(X_s)     \right) ds\\
&\quad+ \int_0^{\tau \wedge \theta_t } e^{-rs} \sigma(X_s)V'(X_s)\,dB_s.
\end{align}
In case \eqref{cond0} holds, we have $V(x)=\frac{K}{r}(1-\varphi_{nK}(x))$. By concavity of $V$ (cf. Lemma~\ref{conc}) we thus obtain $V'(x)\leq V'(0)\leq 1$, so
\begin{align}\label{proof1help1}
\left(V'(x)-1\right)\left( KI_{\{x\geq \hat b\}}-\lambda \right) \leq 0, \mbox{ for all $\lambda \in [0,K]$.}
\end{align}
Similarly, if \eqref{cond} holds, then it follows by concavity of $V$ and $V'(\hat b)=1$ that \eqref{proof1help1} holds. Hence
\begin{align}
V(x) 
& \geq - \int_0^{\tau \wedge \theta_t} e^{-rs} \sigma(X_s)V'(X_s)\,dB_s + e^{-r{\tau \wedge \theta_t}}V(X_{\tau \wedge \theta_t }) + 
\int_0^{\tau \wedge \theta_t} e^{-rs}\lambda(X_s)\,ds.
\end{align}
Taking expectation yields
\begin{align}
V(x) \geq  \mathbb{E}_x\left(e^{-r{\tau \wedge \theta_t }}V(X_{\tau \wedge \theta_t })\right) + 
\mathbb{E}_x\left(\int_0^{\tau \wedge \theta_t } e^{-rs}\lambda(X_s)\,ds\right),
\end{align}
where we used that $V'(\cdot)$ and $\sigma(\cdot)$ are bounded on the stochastic interval $[0,\theta_t]$ so that the expectation of the stochastic integral vanishes. Since $V\geq 0$, using monotone convergence we find that 
\begin{align} \label{asjdashh23214}
V(x) \geq  
\mathbb{E}_x\left(\int_0^{\tau} e^{-rs}\lambda(X_s)\,ds\right) = J_i(x;({\lambda_{i}},\boldsymbol{\hat \lambda}_{-i})).
\end{align}
Furthermore, using the threshold strategy $KI_{\{X_t\geq \hat b\}}$ instead of $\lambda_t$ in the calculations above yields
\begin{align}
V(x) =  \mathbb{E}_x\left(e^{-r{\tau \wedge \theta_t }}V(X_{\tau \wedge \theta_t })\right) + 
\mathbb{E}_x\left(\int_0^{\tau \wedge \theta_t } e^{-rs}\lambda(X_s)\,ds\right),
\end{align}
so $V(0)=0$ and dominated and monotone convergence  (using concavity of $V$ for the domination) yield 
\begin{align} 
V(x)  =   \mathbb{E}_x\left(\int_0^{\tau} e^{-rt}KI_{\{X_t\geq \hat b\}}\,dt\right) =J_i(x;\boldsymbol{\hat \lambda}).
\end{align}
Consequently, $\boldsymbol{\hat \lambda}=(\hat b,...,\hat b)$ is a Nash equilibrium.
%
%
%
%
%
%
%
\end{proof}

\subsection{Properties of the threshold Nash equilibrium}\label{sec:proper-iden-NE}

To understand the effect of competition, we provide a short study of the dependence on the number $n$ of agents for the equilibrium threshold and for the corresponding equilibrium value. Our first result in this direction shows that 
increased competition lowers the equilibrium threshold and decreases each agent's equilibrium value.

\begin{thm} \label{competition} 
\textbf {(The effect of increased competition.)}
With all other parameters being held fixed, the threshold $\hat b=\hat b(n)$ in the symmetric threshold strategy $\boldsymbol{\hat \lambda}=(\hat b,...,\hat b)$ is decreasing in $n$. In fact, there exists a number $\bar n$ such that $\hat b$ is  decreasing for $n < \bar n$ and $\hat b=0$ for $n \geq \bar{n}$.
Moreover, the equilibrium value $V(x)=V(x;n)$ is decreasing in $n$.
\end{thm}

\begin{proof}
First note that $\varphi_{nK}'(0)$ is increasing in $n$, so there exists $\bar n$ such that condition \eqref{cond0} holds if $n\geq \bar n$ (so that $\hat b=0$ for such $n$), and \eqref{cond} holds if $n<\bar n$. Moreover, 
$-\frac{\varphi_{mK}}{\varphi_{mK}'}\leq -\frac{\varphi_{nK}}{\varphi_{nK}'}$ if $m\leq n$. Consequently, the function $f=f_n$ in \eqref{f} is increasing in $n$, so it follows that $\hat b=\hat b(n)$ is decreasing in $n$ for $n<\bar n$ (cf. the proof of Lemma \ref{lem:bhat-is-pos-sol}). 

To show that $V$ is decreasing in $n$, let $m\leq n$ and denote by $\hat b_m,\hat b_n$ and 
$V_m,V_n$ the corresponding thresholds and equilibrium values, respectively. By the above, $\hat b_m\geq \hat b_n$. 

If $\hat b_m\geq \hat b_n>0$, then 
$V_m=A_m\psi$ and $V_n=A_n\psi$ on $[0,\hat b_m]$ and $[0,\hat b_n]$,
respectively, so $V_n'(\hat b_n)=1\leq V_m'(b_n)$ by the concavity of $V_m$. Thus $A_m\geq A_n$, and 
\begin{equation}\label{ineq1}
V_m\geq V_n\mbox{ on }[0,\hat b_n].
\end{equation}
Clearly, \eqref{ineq1} holds also (trivially) in the case $\hat b_n=0$. 

As a consequence of \eqref{ineq1}, we also have $V_m\geq V_n$
on $[0,\hat b_m]$ since $V'_n\leq 1\leq V'_m$ on 
$[\hat b_n,\hat b_m]$. On the interval $[\hat b_m,\infty)$, the function $u:=V_m-V_n$ satisfies
\[\frac{\sigma^2}{2}u''+(\mu-nK)u'-ru=-(n-m)KV_m'\leq 0\]
with boundary conditions $u(\hat b_m)\geq 0$ and $u(\infty)=0$. The maximum principle then yields 
$u\geq 0$ on $[\hat b_m,\infty)$. Consequently, $V_m\geq V_n$ everywhere, which completes the proof.
\end{proof}

Intuitively, the effect in Theorem~\ref{competition}
of adding more agents of the same type is negative as it both adds competition and also increases the maximal total
push rate  $nK$ (while the individual push rate $K$ is constant). In the following result we isolate the effect of increased competition by studying the case with a fixed 
maximal total push rate $\overline K$.

\begin{thm} \label{incrcomp}
\textbf {(The effect of increased competition with constant total extraction rate.)}
Let the maximal individual extraction rate be $K=\overline K/n$ (so that the total maximal extraction rate $\overline K$ is independent of $n$). Then the equilibrium threshold $\hat b=\hat b(n)$ and the total equilibrium value $nV(x)$ are both decreasing in $n$.
\end{thm} 

\begin{proof}
The function $f$ in \eqref{f} is in this case given by $f(b)=\frac{\psi(b)}{\psi'(b)}-\frac{\varphi_{\overline K}(b)}{\varphi_{\overline K}'(b)}$, which is independent of $n$,
and the threshold $\hat b=\hat b(n)$ is given as the unique solution of $f(\hat b)=K/r=\overline K/(rn)$ provided $\varphi_{\overline K}'(0)<-rn/\overline K$, and $\hat b(n)=0$ otherwise. Since $f$ is increasing as long as 
$f<\overline K/r$ (cf. the proof of Lemma \ref{lem:bhat-is-pos-sol}), the solution of $f(\hat b(n))=\overline K/(nr)$ is decreasing in $n$.

It remains to show that $nV(x)$ decreasing. Recall that
\begin{equation}
V(x)=	
\begin{cases}
			D_1(n)\psi(x), 					&  0<x\leq \hat b(n),\\
			D_4(n)\varphi_{{\overline K}}(x) + \frac{{\overline K}}{nr},		& x \geq \hat b(n), 
		\end{cases}
\end{equation}
where $\hat b(n)$ solves
\[f(\hat b(n)):=\frac{\psi(\hat b(n))}{\psi'(\hat b(n))} - \frac{\varphi_{\overline K}(\hat b(n))}{\varphi_{\overline K}'(\hat b(n))}=\frac{\overline K}{nr}\]
if \eqref{cond} holds, and $\hat b(n)=0$ otherwise. Moreover, 
\begin{align}
D_1(n)=\frac{-\frac{\overline K}{nr}\varphi_{{\overline K}}'(\hat b(n))}{\varphi_{{\overline K}}(\hat b(n))\psi'(\hat b(n))-\varphi_{{\overline K}}'(\hat b(n))\psi(\hat b(n))}
\end{align}
and 
\begin{align} 
D_4(n)= 
\frac{-\frac{\overline K}{nr}\psi'(\hat b(n))}{\varphi_{{\overline K}}(\hat b(n))\psi'(\hat b(n)) -\varphi_{{\overline K}}'(\hat b(n))\psi(\hat b(n))}.
\end{align}

In order to show that $nV(x;n)$ is decreasing in $n$ we now show that $nD_1(n)$ and $nD_4(n)$ are decreasing in $n$. We begin with $nD_1(n)$. Since $\hat b(n)$ is decreasing in $n$ (Theorem~\ref{competition}), it suffices to show that $\frac{\varphi_{{\overline K}}'(x)}{\varphi_{{\overline K}}(x)\psi'(x)-\varphi_{{\overline K}}'(x)}$ is decreasing for $x\in(0,\hat b(1))$, which is equivalent to 
$$
\psi''(x)/\psi'(x)-\varphi''_{{\overline K}}(x)/\varphi'_{{\overline K}}(x) \leq 0.
$$
However, 
\begin{align}\label{thm:fixedbarKEQ}
\frac{\sigma^2(x)}{2}(\psi''(x)/\psi'(x)-
\varphi''_{{\overline K}}(x)/\varphi'_{{\overline K}}(x))=rf(x)-\overline K\leq 0
\end{align}
for $x\in(0,\hat b(1))$ (cf. the proof of Lemma \ref{lem:bhat-is-pos-sol}), so $nD_1(n)$ is decreasing in $n$. For $nD_4(n)$ we want to show for $x\in(0,\hat b(1))$, that $\frac{\psi'(x)}{\varphi_{{\overline K}}(x)\psi'(x) -\varphi_{{\overline K}}'(x)\psi(x)}$ is decreasing, which is equivalent to 
$$\psi'(x)\varphi_{{\overline K}}''(x) -\varphi_{{\overline K}}'(x)\psi''(x)\leq 0.$$
Hence, also in this case \eqref{thm:fixedbarKEQ} gives us what is needed. 

Now, since $nD_1(n)$ and $nD_4(n)$ are decreasing in $n$, it follows that $mV(x;m)\geq nV(x;n)$ for 
$x\in[0,\hat b(n)]\cup [\hat b(m),\infty)$ if $m<n$. Moreover, on $[\hat b(n),\hat b(m)]$ we have that the function
$H(x):=mV(x;m)-nV(x;n)$ satisfies 
\[\frac{\sigma^2}{2}H''+(\mu-\overline K) H'-rH= -\overline K(mV'(x;m)-1)\leq 0\]
since $V'(x;m)\geq 1$ on $[\hat b(n),\hat b(m)]$. Since $H>0$ at the boundary points $\{\hat b(n),\hat b(m)\}$, a maximum principle gives $H\geq 0$, so $mV(x;m)\geq nV(x;n)$ everywhere.
\end{proof}

\begin{rem} \label{rem:non-mon}
In the single-agent game, the optimal threshold (given by $\hat b=\hat b(K)$ determined with $n=1$) can be shown to be increasing as a function of the maximal extraction rate $K$. Indeed, implicit differentiation of the relation
$f(\hat b(K))=K/r$ yields
\[\hat b'(K)=\frac{\sigma^2(\hat b)\psi'(\hat b)}{2K(\mu(\hat b)\psi'(\hat b)-r\psi(\hat b))}\geq 0,\]
where the inequality is a consequence of $\psi''(\hat b)\leq 0$.
In the competitive setting of the current paper (where we may have $n\geq 2$), however, the effect of an increase in the extraction rate is ambiguous as it facilitates fast extraction not only for you but also for your competitors. We therefore do not expect any monotone relationship between $\hat b$ and $K$; see Figure \ref{sym:fig3} for an illustration of this relationship in the case of constant coefficients. 
\end{rem}
 
 
\subsection{The case of constant coefficients}\label{sec:constant-coeff-sym}
Consider the case of constant drift and  diffusion coefficients $\mu>0$ and $\sigma>0$.  In this case
\begin{align}
\psi(x) = \frac{e^{\alpha x}- e^{ \beta x}}{\alpha-\beta} 
\enskip \mbox{ and } \enskip
\varphi_{nK}(x)=e^{\gamma x}
\end{align}
where 
$\alpha:= -\frac{\mu}{\sigma^2} + \sqrt{\frac{\mu^2}{\sigma^4}+\frac{2r}{\sigma^2}}>0$, 
$\beta:= -\frac{\mu}{\sigma^2} - \sqrt{\frac{\mu^2}{\sigma^4}+\frac{2r}{\sigma^2}}<0$
and 
$\gamma:=-\frac{\mu-nK}{\sigma^2} - \sqrt{\frac{(\mu-nK)^2}{\sigma^4}+\frac{2r}{\sigma^2}}<0$.
Let us use Theorem \ref{verifi-thm} to find the Nash equilibrium for this model. First note that condition \eqref{cond} is equivalent to
\begin{align} \label{cond:example1}
\frac{K}{r}\left(\frac{\mu-nK}{\sigma^2} +\sqrt{\frac{(\mu-nK)^2}{\sigma^4}+\frac{2r}{\sigma^2}}\right)>1.
\end{align}
Hence, if \eqref{cond:example1} does not hold, then $\hat b=0$ is the equilibrium threshold, and if \eqref{cond:example1} holds, then the equilibrium threshold is strictly positive and determined by the equation $\frac{e^{\alpha \hat b }- e^{\beta \hat b }}{\alpha e^{\alpha \hat b }- \beta e^{\beta \hat b }} - \frac{1}{\gamma} =\frac{K}{r}$, cf. \eqref{b}. In the latter case, we find
\begin{align}
\hat b = \frac{\ln{\left(\frac{1-(K/r+1/\gamma)\beta}{1-(K/r+1/\gamma)\alpha}\right)}}{\alpha-\beta}.
\end{align}
Moreover, \eqref{eqvalbhatnot0} yields a closed formula for the equilibrium value function $V$. Figures \ref{sym:fig1}--\ref{sym:fig3}, illustrate the monotonicities in $K$ and $n$ obtained in Theorem \ref{competition} and Theorem \ref{incrcomp}, as well as the non-monotonicity of $\hat b$ in $K$, cf. Remark \ref{rem:non-mon}. The remaining parameters are held fixed, with
					$\mu=4$, 
					$\sigma^2=2$  
					and $r=0.05$.

\begin{figure}[H]
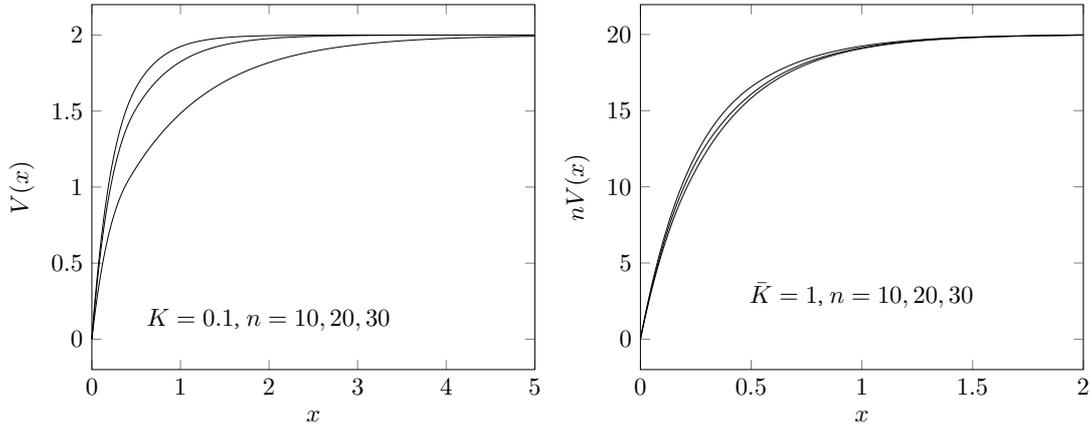

	\center
 
		\caption{Left: The individual equilibrium value function $V(x)$ (same for each agent) when the number of competitors $n$ is varied and the individual maximal extraction rate $K$ is fixed.
\newline
 Right: The total equilibrium value function $nV(x)$ when $n$ is varied and the total maximal extraction rate $\bar K$ is  fixed. (Larger $n$ correspond to lower equilibrium values.)} \label{sym:fig1}
\end{figure} 


\begin{figure}[H]
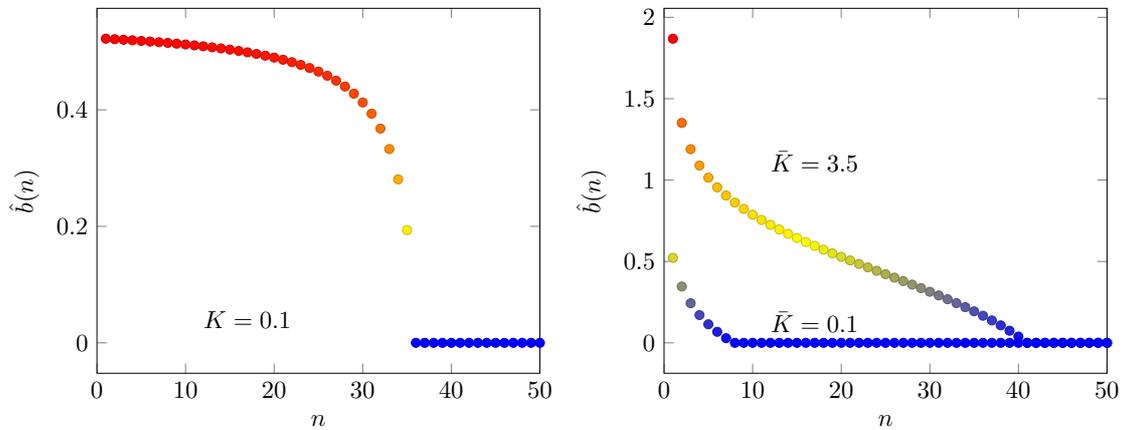

	\center
%
\caption{
The equilibrium threshold $\hat b$ as a function of the number of competitors $n$ when the individual maximal extraction rate $K$ is fixed (left), and when the total maximal extraction rate $\bar K$ is fixed (right), respectively.} \label{sym:fig2}
\end{figure}


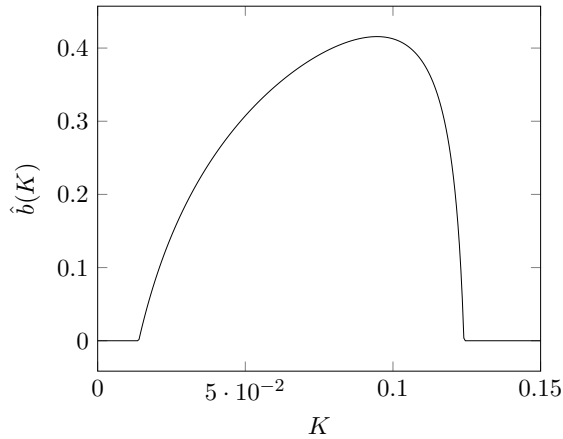
\begin{figure}[H]
\center
\begin{tikzpicture}[scale=0.85] 
		\begin{axis}[xmin=0, xmax=0.15,xlabel={$K$}, ylabel={$\hat b(K)$}, xtick distance = 0.05]
				\addplot[style=solid]		
		coordinates{
(0,0)
(0.0005,0)
(0.001,0)
(0.0015,0)
(0.002,0)
(0.0025,0)
(0.003,0)
(0.0035,0)
(0.004,0)
(0.0045,0)
(0.005,0)
(0.0055,0)
(0.006,0)
(0.0065,0)
(0.007,0)
(0.0075,0)
(0.008,0)
(0.0085,0)
(0.009,0)
(0.0095,0)
(0.01,0)
(0.0105,0)
(0.011,0)
(0.0115,0)
(0.012,0)
(0.0125,0)
(0.013,0)
(0.0135,0)
(0.014,0.00174559715589583)
(0.0145,0.0103719012334778)
(0.015,0.0187015245648502)
(0.0155,0.026753841832146)
(0.016,0.0345463802389636)
(0.0165,0.04209504695137)
(0.017,0.0494143225751019)
(0.0175,0.0565174265788691)
(0.018,0.0634164594076385)
(0.0185,0.0701225251194855)
(0.019,0.0766458376636224)
(0.0195,0.0829958133500885)
(0.02,0.0891811516093912)
(0.0205,0.0952099057775575)
(0.021,0.101089545349197)
(0.0215,0.106827010903482)
(0.022,0.112428762713979)
(0.0225,0.117900823894177)
(0.023,0.123248818799472)
(0.0235,0.12847800729781)
(0.024,0.133593315430964)
(0.0245,0.138599362913026)
(0.025,0.143500487849565)
(0.0255,0.148300769007696)
(0.026,0.153004045922423)
(0.0265,0.157613937086557)
(0.027,0.162133856439147)
(0.0275,0.166567028339723)
(0.028,0.170916501192064)
(0.0285,0.175185159860869)
(0.029,0.179375737007279)
(0.0295,0.183490823454109)
(0.03,0.187532877678608)
(0.0305,0.191504234519237)
(0.031,0.19540711317312)
(0.0315,0.199243624552212)
(0.032,0.203015778058748)
(0.0325,0.206725487833945)
(0.033,0.210374578528166)
(0.0335,0.213964790635666)
(0.034,0.217497785432586)
(0.0345,0.220975149552886)
(0.035,0.224398399233417)
(0.0355,0.227768984256246)
(0.036,0.231088291613567)
(0.0365,0.2343576489181)
(0.037,0.237578327579697)
(0.0375,0.240751545766909)
(0.038,0.243878471170539)
(0.0385,0.246960223584648)
(0.039,0.249997877319064)
(0.0395,0.2529924634562)
(0.04,0.255944971963862)
(0.0405,0.258856353674688)
(0.041,0.261727522141964)
(0.0415,0.264559355380716)
(0.042,0.267352697502237)
(0.0425,0.270108360249513)
(0.043,0.272827124440428)
(0.0435,0.275509741325024)
(0.044,0.278156933862623)
(0.0445,0.280769397924129)
(0.045,0.283347803424423)
(0.0455,0.285892795389364)
(0.046,0.288404994961581)
(0.0465,0.290885000348885)
(0.047,0.293333387718867)
(0.0475,0.295750712042977)
(0.048,0.298137507893095)
(0.0485,0.300494290193432)
(0.049,0.30282155493035)
(0.0495,0.305119779822512)
(0.05,0.307389424953593)
(0.0505,0.309630933369631)
(0.051,0.311844731642923)
(0.0515,0.314031230404251)
(0.052,0.316190824845089)
(0.0525,0.318323895191317)
(0.053,0.320430807149853)
(0.0535,0.322511912329532)
(0.054,0.324567548637433)
(0.0545,0.326598040651791)
(0.055,0.328603699972528)
(0.0555,0.33058482555038)
(0.056,0.332541703995489)
(0.0565,0.334474609866303)
(0.057,0.336383805939523)
(0.0575,0.33826954346179)
(0.058,0.340132062383753)
(0.0585,0.341971591577098)
(0.059,0.343788349035051)
(0.0595,0.345582542056836)
(0.06,0.347354367416532)
(0.0605,0.349104011516671)
(0.061,0.350831650526955)
(0.0615,0.352537450508355)
(0.062,0.35422156752285)
(0.0625,0.355884147729021)
(0.063,0.357525327463664)
(0.0635,0.359145233309548)
(0.064,0.360743982149412)
(0.0645,0.362321681206255)
(0.065,0.363878428069922)
(0.0655,0.365414310709966)
(0.066,0.366929407474725)
(0.0665,0.368423787076491)
(0.067,0.369897508562644)
(0.0675,0.371350621272556)
(0.068,0.372783164780028)
(0.0685,0.374195168821003)
(0.069,0.375586653206223)
(0.0695,0.37695762771847)
(0.07,0.378308091993983)
(0.0705,0.379638035387575)
(0.071,0.380947436820942)
(0.0715,0.382236264613587)
(0.072,0.383504476295724)
(0.0725,0.384752018402466)
(0.073,0.385978826248539)
(0.0735,0.387184823682679)
(0.074,0.388369922820819)
(0.0745,0.389534023757052)
(0.075,0.390677014251328)
(0.0755,0.391798769392682)
(0.076,0.392899151236755)
(0.0765,0.393978008416217)
(0.077,0.395035175722607)
(0.0775,0.396070473657991)
(0.078,0.397083707954676)
(0.0785,0.398074669061094)
(0.079,0.399043131591823)
(0.0795,0.399988853739508)
(0.08,0.400911576646306)
(0.0805,0.401811023732228)
(0.081,0.40268689997759)
(0.0815,0.403538891156494)
(0.082,0.404366663018048)
(0.0825,0.405169860411704)
(0.083,0.405948106352843)
(0.0835,0.406701001024347)
(0.084,0.407428120709546)
(0.0845,0.408129016651552)
(0.085,0.408803213833492)
(0.0855,0.409450209673717)
(0.086,0.4100694726295)
(0.0865,0.410660440702142)
(0.087,0.411222519835782)
(0.0875,0.411755082201442)
(0.088,0.41225746435709)
(0.0885,0.4127289652736)
(0.089,0.413168844215509)
(0.0895,0.413576318464403)
(0.09,0.413950560871568)
(0.0905,0.414290697225182)
(0.091,0.414595803415863)
(0.0915,0.414864902382721)
(0.092,0.415096960820189)
(0.0925,0.415290885623873)
(0.093,0.415445520051304)
(0.0935,0.415559639570884)
(0.094,0.415631947369383)
(0.0945,0.415661069485056)
(0.095,0.415645549529706)
(0.0955,0.415583842958842)
(0.096,0.415474310844296)
(0.0965,0.415315213098289)
(0.097,0.415104701091784)
(0.0975,0.414840809602991)
(0.098,0.414521448023919)
(0.0985,0.414144390743743)
(0.099,0.413707266617328)
(0.0995,0.413207547415234)
(0.1,0.412642535137699)
(0.1005,0.412009348059159)
(0.101,0.411304905351344)
(0.1015,0.410525910111585)
(0.102,0.409668830597946)
(0.1025,0.40872987944366)
(0.103,0.407704990589179)
(0.1035,0.406589793629993)
(0.104,0.405379585231063)
(0.1045,0.404069297202683)
(0.105,0.40265346076609)
(0.1055,0.401126166457805)
(0.106,0.399481019026806)
(0.1065,0.397711086564453)
(0.107,0.395808842969325)
(0.1075,0.393766102681929)
(0.108,0.391573946420425)
(0.1085,0.389222636398774)
(0.109,0.38670151920096)
(0.1095,0.383998914103595)
(0.11,0.381101984163803)
(0.1105,0.377996586792692)
(0.111,0.374667099781156)
(0.1115,0.371096217785808)
(0.112,0.367264713053408)
(0.1125,0.363151152572916)
(0.113,0.358731561772067)
(0.1135,0.353979022148277)
(0.114,0.348863186598884)
(0.1145,0.343349691346038)
(0.115,0.337399436733887)
(0.1155,0.330967700070204)
(0.116,0.324003030986365)
(0.1165,0.316445861819841)
(0.117,0.308226739683889)
(0.1175,0.299264049075824)
(0.118,0.289461037459307)
(0.1185,0.278701870262015)
(0.119,0.266846307480771)
(0.1195,0.253722378805197)
(0.12,0.239116078313336)
(0.1205,0.222756490841346)
(0.121,0.204293677652746)
(0.1215,0.183264625666437)
(0.122,0.159038575627008)
(0.1225,0.130724641665521)
(0.123,0.0970054273884855)
(0.1235,0.0558116565336994)
(0.124,0.00361140691020194)
(0.1245,0)
(0.125,0)
(0.1255,0)
(0.126,0)
(0.1265,0)
(0.127,0)
(0.1275,0)
(0.128,0)
(0.1285,0)
(0.129,0)
(0.1295,0)
(0.13,0)
(0.1305,0)
(0.131,0)
(0.1315,0)
(0.132,0)
(0.1325,0)
(0.133,0)
(0.1335,0)
(0.134,0)
(0.1345,0)
(0.135,0)
(0.1355,0)
(0.136,0)
(0.1365,0)
(0.137,0)
(0.1375,0)
(0.138,0)
(0.1385,0)
(0.139,0)
(0.1395,0)
(0.14,0)
(0.1405,0)
(0.141,0)
(0.1415,0)
(0.142,0)
(0.1425,0)
(0.143,0)
(0.1435,0)
(0.144,0)
(0.1445,0)
(0.145,0)
(0.1455,0)
(0.146,0)
(0.1465,0)
(0.147,0)
(0.1475,0)
(0.148,0)
(0.1485,0)
(0.149,0)
(0.1495,0)
(0.15,0)
};
		\end{axis} 
	\end{tikzpicture} 
%
%
%
%
%
%
%
		%
\caption{The equilibrium threshold $\hat b$ as a function of the individual extraction rate $K$, with the number of competitors fixed at $n=30$. 
} \label{sym:fig3}
\end{figure}

%
%
%
%
%
%
%
%
%
%
%


\section{Individual extraction rates}\label{sec:Individual}

In this section we consider the asymmetric case when maximal extraction rates are not necessarily identical. In particular, a Markovian control 
$\lambda_{it}=\lambda_i(X_t)$ is in this section an admissible strategy for agent $i$ if it takes values in $[0,K_i]$, where $K_1,...,K_n$ are given positive constants.
Again we will search for Nash equilibria of threshold type, which in this section means that
for all $i=1,...,n$ we have
\begin{align} 
\lambda_{it}=K_iI_{\{ X_t\geq b_i\}}
\enskip \mbox{for some constant $b_i\geq0$ and all $t\geq0$.}
\end{align}
In Section \ref{one-agent-problem:sec} we report a verification type result for the asymmetric game. 
In Section~\ref{oneNE-nonidentical:sec} we establish the existence of a Nash equilibrium of threshold type using a fixed-point argument. The corresponding thresholds are shown to be ordered in the same way as the maximal extraction rates.


\subsection{A single-agent problem and a verification result} \label{one-agent-problem:sec}
By definition a threshold strategy $(\hat b_1,...,\hat b_n)$ is Nash equilibrium if the threshold strategy corresponding to $\hat b_i$ is optimal (in the usual sense) for any agent $i$ in case the other agents use the fixed threshold strategy $(\hat b_1,...,\hat b_{i-1},\hat b_{i+1},...,\hat b_n)$ (we use this notation to mean the threshold strategy obtained when removing the threshold of agent $i$ also in case $i=1$ or $i=n$). 
Thus, all we need in order to verify that a proposed threshold equilibrium is indeed an equilibrium is to solve a single-agent version of the problem investigated in the previous section, but with the complication of a drift function with negative jumps; in particular, the problem of selecting an admissible strategy $\lambda_{i}=(\lambda_{it})_{t\geq 0}$ for agent $i$ in 
\begin{align}\label{stateprocesszyzyzy}
dX_t =  \left(\mu(X_t)- \sum_{j\neq i}K_jI_{\{ X_t\geq b_j\}}- \lambda_{nt}\right)dt + \sigma(X_t) dB_t, 
\end{align}
in order to maximize the reward 
\begin{align}  \label{stateprocesszyzyzy2}
\mathbb{E}_x\left(\int_0^{\tau} e^{-rt}{\lambda_{it}}\,dt\right), 
\end{align}
for fixed constants $b_1,...,b_{i-1},b_{i+1},...,b_n\in[0,\infty), K_1,...,K_n \in (0,\infty)$. 




In line with Section~\ref{sec:problem:ident}, we thus denote by $\psi^{(i)}(x)=\psi^{(i)}(x;b_1,...,b_{i-1},b_{i+1},...,b_n)$ the increasing positive solution of 
\begin{align}\label{h-ode-with-kinks}
& \frac{1}{2}\sigma^2(x)h''(x)+\left(\mu(x)-\sum_{j \neq i}K_jI_{\{x\geq b_j\}}\right)h'(x) 
-rh(x)=0,
\end{align}
and by $\varphi^{(i)}(x)=\varphi^{(i)}(x;b_1,...,b_{i-1},b_{i+1},...,b_n)$ the decreasing positive solution of
\begin{align}\label{h-ode-with-kinks2}
& \frac{1}{2}\sigma^2(x)h''(x)+\left(\mu(x)-K_i-\sum_{j\neq i}K_jI_{\{x\geq b_j\}}\right)h'(x) 
-rh(x)=0,
\end{align}
with the boundary specifications $\psi^{(i)}(0)=\varphi^{(i)}(\infty)=0$, ${\psi^{(i)}}'(0)=1$ and $\varphi^{(i)}(\infty)(0)=1$. Note that these functions are analogous to $\psi$ and $\varphi_{nK}$ of the previous section, but due to the discontinuities of the drift function we now  
merely have that
$\psi^{(i)},\varphi^{(i)}\in \mathcal C^1(0,\infty) \cap \mathcal C^2\left([0,\infty)\backslash\{b_1,...,b_{i-1},b_{i+1},...,b_n\}\right)$.
However, since the jumps of the drift are negative, we may use approximation by 
drift coefficients for which \ref{coeff-assum:1} and \ref{coeff-assum:3} are valid.
More specifically, 
consider drift coefficients $\{\mu_k\}_{k=1}^\infty$ satisfying \ref{coeff-assum:1} and \ref{coeff-assum:3}, and with $\mu_k\uparrow\mu-\sum_{j \neq i}K_jI_{\{\cdot\geq b_j\}}$ pointwise.
We then see that the conclusions of Lemma~\ref{lemma1} still hold. In fact, the corresponding fundamental solutions converge pointwise from above to 
$\psi^{(i)}$ and $\varphi^{(i)}$, as $k\to\infty$. 
Using that the pointwise limit of concave/convex functions is again concave/convex, 
Lemma~\ref{lemma1} extends as follows.

\begin{lem}\label{lemma2:nonidentical} 
\begin{enumerate} [label=(\roman*)]
\item
Each function $\psi^{(i)}$ is (strictly) concave-convex on $[0,\infty)$ with a unique inflection point $b^{**}_i\in[0,\infty)$. In particular, if  $\mu(0)-\sum_{j\neq i}K_jI_{\{ b_j=0\}}>0$ then $b_i^{**}>0$, and otherwise $b_i^{**}=0$ and $\psi^{(i)}$ is convex everywhere. 
\item
Each function $\varphi^{(i)}$ is convex.
\end{enumerate}
\end{lem}

\begin{rem}\label{rem:identify-f-g}
In practice, each $\psi^{(i)}$ can be found by 
(i) identifying the fundamental increasing and decreasing solutions to \eqref{h-ode-with-kinks} separately on the intervals on which the drift function does not jump, 
(ii) on each interval forming linear combinations of the increasing and decreasing fundamental solutions,  
(iii) pasting the linear combinations together to form a function $\mathbb{R}_+\rightarrow \mathbb{R}$,   
(iv) choosing the constants of the linear combinations so that the resulting $\mathbb{R}_+\rightarrow \mathbb{R}$ function is continuously differentiable and the boundary conditions are satisfied; this function is then $\psi^{(i)}$. We can similarly obtain $\varphi^{(i)}$. 
\end{rem}

\begin{rem}
As in Remark~\ref{remark-optimal-div}, the inflection point $b^{**}_i$ corresponds to the optimal barrier policy in case the maximal push rate of agent $i$ is $K_i=\infty$.
\end{rem}

Relying on Lemma \ref{lemma2:nonidentical}, the reasoning in Section~\ref{sec:problem:ident} extends to the present case. 
In particular, by relying on a single-agent version of Theorem \ref{sec:verthm1} for the present setting, we can verify if it is optimal for any agent $i$ to comply with a proposed threshold equilibrium. We thus obtain the following verification type result. 

\begin{thm} \label{veri-thm-asym} Suppose a threshold strategy $(\hat b_1,..., \hat b_n)$ satisfies $V_i'(\hat b_i) = 1$ if $\hat b_i >0$ and 
$V_i'(\hat b_i) \leq 1$ if $\hat b_i =0$, 
for $i=1,...,n$, where
\begin{equation} \label{value-funcXz}
V_i(x) := 
\begin{cases} 
E_{1,i} \psi^{(i)}(x) 	&  0\leq x \leq \hat b_i \\
E_{4,i} \varphi^{(i)}(x) + \frac{K_i}{r}					&  x\geq \hat b_i ,\\
\end{cases}
\end{equation}
and  
\begin{align}\label{E}
E_{1,i}&:= \frac{-\frac{K_i}{r}{\varphi^{(i)}}'(\hat b_i)}{\varphi^{(i)}(\hat b_i){\psi^{(i)}}'(\hat b_i)-{\varphi^{(i)}}'(\hat b_i){\psi^{(i)}}(\hat b_i)},\\
E_{4,i}&:=  \frac{-\frac{K_i}{r}{\psi^{(i)}}'(\hat b_i)}{\varphi^{(i)}(\hat b_i){\psi^{(i)}}'(\hat b_i)-{\varphi^{(i)}}'(\hat b_i){\psi^{(i)}}(\hat b_i)}.
\end{align}
Then, $(\hat b_1,..., \hat b_n)$ is a threshold Nash equilibrium and $V_i(x),i=1,...,n$ are the corresponding equilibrium value functions.  
\end{thm}


We will now report a boundedness and continuity result that is essential for the fixed-point argument in Section \ref{oneNE-nonidentical:sec}. 
Consider the problem stated in the beginning of this section, see \eqref{stateprocesszyzyzy}--\eqref{stateprocesszyzyzy2}. Relying again on Lemma \ref{lemma2:nonidentical} and reasoning similar to that of Section~\ref{sec:problem:ident} we find, as noted above, that the solution is a threshold strategy. 
In particular, if agents $1,...,i-1,i+1,...,n$ use a fixed strategy $(b_1,...,b_{i-1},b_{i+1},..., b_n)$, then the optimal strategy of agent $i$ is a threshold strategy, which we denote by $b^Z=b^Z(b_1,...,b_{i-1},b_{i+1},..., b_n)$. This induces mappings
\begin{align}\label{optimalresponsemapping}
(b_1,...,b_{i-1},b_{i+1},..., b_n) \in \mathbb{R}^{n-1}_+ \mapsto b_i^Z=b^Z_i(b_1,...,b_{i-1},b_{i+1},..., b_n),
\end{align}
for $i=1,...,n$. The interpretation is that  $b_i^Z$ is the optimal threshold strategy for agent $i$ under the assumption that the other agents use a fixed threshold strategy $(b_1,...,b_{i-1},b_{i+1},..., b_n)$. We remark that a more explicit interpretation of the mappings $b_i^Z$ can be extracted from the proof below.

\begin{prop} \label{b-Zbounded} For each $i=1,...,n$, the mapping \eqref{optimalresponsemapping}
is bounded by $c$ and continuous.
\end{prop}

\begin{proof} Suppose without loss of generality that $i=n$. Let us first show that the function 
$(b_1,...,b_{n-1}) \mapsto b^{**}_n$ is bounded. Consider any $x>0$ with $x \neq b_i$ for all $i<n$ and suppose ${\psi^{(n)}}''(x)<0$. The function ${\psi^{(n)}}(y) - y{\psi^{(n)}}'(y)$ is zero for $y=0$ and increasing on $[0,x]$, so
${\psi^{(n)}}(x) \geq x{\psi^{(n)}}'(x)$ for such $x$. This implies that 
\begin{align}
\frac{1}{2}\sigma^2(x){\psi^{(n)}}''(x) 
& =  r{\psi^{(n)}}(x) - \left(\mu(x)-\sum_{i\leq n-1}K_iI_{\{x\geq b_i\}}\right){\psi^{(n)}}'(x)\\
& \geq \left(rx- \mu(x) + \sum_{i\leq n-1}K_iI_{\{x\geq b_i\}}\right){\psi^{(n)}}'(x).
\end{align}
By Assumption \ref{coeff-assum}\ref{coeff-assum:3}, the last expression above is positive whenever $x\geq c$; thus, in order for ${\psi^{(n)}}''(x)<0$ to hold we must have $x<c$. Consequently, $b_n^{**}\leq c$, and since $c$ does not depend on $b_1,...,b_{n-1}$, it follows that $(b_1,...,b_{n-1})  \mapsto b^{**}$ is bounded. 

It can now be seen (cf. the single-agent, $n=1$, version of Theorem \ref{sec:verthm1}), that the optimal value function corresponding to 
$b_n^Z=b^Z_n(b_1,...,b_{n-1})$ is 
\begin{equation} \label{non-indenticalSCprop2:valueF} 
Z(x):=
\begin{cases}
E_1 {\psi^{(n)}}(x)			&  0\leq x \leq b_n^Z \\
E_4 {\varphi^{(n)}}(x) + \frac{K_n}{r}					&  x\geq b_n^Z ,\\
\end{cases}
\end{equation} 
(where $E_1$ and $E_4$ are defined so that $Z$ is ${\cal C}^1$ at $b_n^Z$, i.e. analogously to the constants in Theorem \ref{veri-thm-asym}) and $b_n^Z \in (0,b_n^{**}]$ is the unique threshold satisfying $Z'(b^Z_n)=1$ 
in case 
\begin{align} \label{non-indenticalSCprop2:condition} 
{\varphi^{(n)}}'(0) < -\frac{r}{K_n}
\end{align}
and $b^{**}_n>0$ hold, and $b^Z_n=0$ otherwise. We thus find 
(i) since $b_n^Z\leq b^{**}_n$, it holds that $b_n^Z$ is bounded, and 
(ii) since the solutions ${\psi^{(n)}}$ and ${\varphi^{(n)}}$ to the ODEs \eqref{h-ode-with-kinks} and \eqref{h-ode-with-kinks2} (with $i=n$) depend continuously on the parameters $b_1,...,b_{n-1}$, it holds that the continuity of $b^Z_n$ can be obtained from the explicit defining relations above, see e.g. \eqref{non-indenticalSCprop2:valueF}. 
\end{proof}

\subsection{Existence of asymmetric threshold Nash equilibrium}\label{oneNE-nonidentical:sec}


\begin{thm} \label{thm:orderedNEasym}  Suppose, without loss of generality, that $K_n \geq... \geq K_1\geq0$. Then, there exists a threshold type Nash equilibrium $\boldsymbol{\hat \lambda}= (\hat b_1,...,\hat b_n)$
that satisfies $\hat b_n\geq...\geq\hat b_1\geq0$. Moreover, the corresponding equilibrium value functions are given by \eqref{value-funcXz} and satisfy $V_n(x)\geq... \geq V_1(x)$ 
for all $x\geq0$.
\end{thm}

\begin{proof}  
For any fixed $i$, the mapping $b_i^Z$ defined in \eqref{optimalresponsemapping} produces the optimal threshold strategy for agent $i$ under the assumption that the other agents use a fixed threshold strategy $(b_1,...,b_{i-1},b_{i+1},..., b_n)$. 
Recall that each $b_i^Z$ is a bounded and continuous mapping by Proposition \ref{b-Zbounded}; in fact, $b^Z_i\leq c$ for all $i$ and $(b_1,...,b_n)$. 
Define the hyper-cube $\mathcal K=\left[0,c\right]^n$ and the composite mapping 
\begin{align} \label{composite-mapping}
\mathcal K\ni (b_1,...,b_n) \mapsto (b^Z_1,...,b^Z_n).
\end{align}
Note that this is a continuous mapping from $\mathcal K$ to $\mathcal K$ and with Brouwer's fixed-point theorem we conclude that it has a fixed-point $(\hat b_1,...,\hat b_n)$. Clearly, the
threshold strategy profile $(\hat b_1,...,\hat b_n)$ is then a Nash equilibrium and the corresponding equilibrium value functions are given by \eqref{value-funcXz}.

For the ordering result, it suffices to consider the case $n=2$ since the strategies of the remaining agents can be considered already locked in. 
Assume that $K_1\leq K_2$ and consider a threshold type equilibrium $(\hat b_1, \hat b_2)$ as obtained above.
%
%
Now assume (to reach a contradiction) that $\hat b_1> \hat b_2$. If $\hat b_2>0$, then 
$$
V_i(x) = A_i\psi(x), \enskip x \leq \hat b_2.
$$
Using strict concavity of $\psi$ (on the interval $(0,\hat b_2)$) and $V_i'(\hat b_i)=1$ we obtain $A_1>A_2>0$, so 
$V_1(\hat b_2)> V_2(\hat b_2)$.
Moreover, $V_i'(x)\geq 1$ for $x\leq \hat b_i$ and $V_i'(x)\leq 1$ for $x\geq \hat b_i$, so
\begin{equation}
\label{atb1}
V'_1(\hat b_1) > V'_2(\hat b_1)\mbox{ and }V_1(\hat b_1) > V_2(\hat b_1).
\end{equation}
A similar reasoning shows that \eqref{atb1} also holds in the case $0=\hat b_2<\hat b_1$.
Define 
\[x_0=\inf\{x\geq \hat b_1:V'_1(\hat b_1) - V'_2(\hat b_1)\leq 0\}.\]
If $x_0<\infty$, then $V_1-V_2\geq 0$, $V_1'-V_2'=0$ and $V''_1-V''_2\leq 0$  at $x_0$. Consequently,
\begin{eqnarray}
0 &=& \frac{\sigma^2}{2}(V_1''-V_2'')+(\mu-K_1-K_2)(V'_1-V_2')-r(V_1-V_2)+(K_1-K_2)\\
&\leq& -r(V_1-V_2)+(K_1-K_2)<0.
\end{eqnarray}
Thus $x_0=\infty$, so $V_1(x)-V_2(x)\geq V_1(\hat b_1)-V_2(\hat b_1)>0$ for all $x\geq \hat b_1$, which contradicts $V_1(\infty)=\frac{K_1}{r}\leq \frac{K_2}{r}= V_2(\infty)$. It follows that $\hat b_1\leq \hat b_2$.

Now that $\hat b_1\leq \hat b_2$, a similar reasoning as above shows that $V_1\leq V_2$ on $(0,\hat b_2)$. Since 
\[\frac{\sigma^2}{2}V_1''+(\mu-K_1-K_2)V'_1-rV_1+K_2\geq \frac{\sigma^2}{2}V_2''+(\mu-K_1-K_2)V'_2-rV_2+K_2=0\]
on $(\hat b_2,\infty)$ and $V_1\leq V_2$ for $x\in\{\hat b_2,\infty\}$, the maximum principle shows that $V_1\leq V_2$ everywhere.
%
%
%
%
\end{proof}

%
%
%

\subsection{The case of constant coefficients}\label{sec:const-coeff-asym}
We consider two agents $i=1,2$ with maximal extraction rates $K_2 \geq K_1 >0$ and constant drift and diffusion coefficients $\mu>0$ and $\sigma>0$. Then existence of a threshold equilibrium $(\hat b_1,\hat b_2)$ satisfying $\hat b_2 \geq \hat b_1 \geq 0$ follows from Theorem \ref{thm:orderedNEasym}. Hence, an equilibrium of either of the following kinds can be found
(i) $(\hat b_1,\hat b_2)=(0,0)$, 
(ii) $(0,\hat b_2)$ with $\hat b_2>0$, and 
(iii) $(\hat b_1,\hat b_2)$ with $\hat b_2\geq \hat b_1>0$. We will now describe how to determine if an equilibrium of the last kind exists and how it can be found. Equilibria of the other kinds can be investigated similarly.  

Using the program in Remark \ref{rem:identify-f-g} we find that 
\begin{align} 
{\psi^{(1)}}(x)  & = \frac{e^{\alpha x}- e^{ \beta x}}{\alpha-\beta}, \enskip 0\leq x \leq    b_1\\
{\psi^{(2)}}(x) & = 
\left\{\begin{array}{l}
\frac{e^{\alpha x}- e^{ \beta x}}{\alpha-\beta}, \enskip 0\leq x \leq    b_1 \\
F_1 e^{\alpha_1 x}- F_2 e^{ \beta_1 x}, \enskip b_1 \leq x \leq    b_2,
\end{array} \right.\\
{\varphi^{(1)}}(x) & =
\left\{\begin{array}{l}
F_3 e^{\alpha_1 x}- (F_3-1)e^{ \beta_1x}, \enskip 0\leq x \leq    b_2 \\
F_4e^{ \beta_2 x} 												, \enskip x \geq    b_2
\end{array} \right.\\
{\varphi^{(2)}}(x) & = C {\varphi^{(1)}}(x), \enskip x\geq    b_2
\end{align}
for some constant $C>0$ implying that the boundary condition ${\varphi^{(2)}}(0)=1$ is satisfied (there is no 
need to determine these functions outside the specified intervals and the constant $C$ will cancel in the relevant calculations);  
here $\alpha:= -\frac{\mu}{\sigma^2} + \sqrt{\frac{\mu^2}{\sigma^4}+\frac{2r}{\sigma^2}}$ and
$\beta:= -\frac{\mu}{\sigma^2} - \sqrt{\frac{\mu^2}{\sigma^4}+\frac{2r}{\sigma^2}}$
as in Section~\ref{sec:constant-coeff-sym}, 
$\alpha_1:= -\frac{\mu-K_1}{\sigma^2} + \sqrt{\frac{(\mu-K_1)^2}{\sigma^4}+\frac{2r}{\sigma^2}}$, 
$\beta_1:= -\frac{\mu-K_1}{\sigma^2} - \sqrt{\frac{(\mu-K_1)^2}{\sigma^4}+\frac{2r}{\sigma^2}}$
and
$\beta_2:= -\frac{\mu-K_1-K_2}{\sigma^2} - \sqrt{\frac{(\mu-K_1-K_2)^2}{\sigma^4}+\frac{2r}{\sigma^2}}$, and the constants $F_i$ (which depend on $b_1$ and $b_2$) are determined so that the functions are continuously differentiable, i.e. according to the linear equation systems
\begin{equation}
\label{eqsys1}
\left\{\begin{array}{l}
e^{\alpha_1 b_1}F_1- e^{ \beta_1 b_1}F_2   = \frac{e^{\alpha b_1}- e^{ \beta b_1}}{\alpha-\beta} \\
\alpha_1 e^{\alpha_1 b_1}F_1 - \beta_1e^{ \beta_1 b_1}F_2   = \frac{\alpha e^{\alpha b_1}- \beta e^{ \beta b_1}}{\alpha-\beta}  
\end{array}\right.
\end{equation} 
and
\begin{equation}
\label{eqsys2}
\left\{
\begin{array}{l}
(e^{\alpha_1b_2}- e^{ \beta_1b_2})F_3  - e^{\beta_2 b_2}F_4   = - e^{ \beta_1b_2} \\
(\alpha_1  e^{\alpha_1b_2}- \beta_1e^{\beta_1b_2})F_3 -  \beta_2 e^{\beta_2 b_2}F_4 = - \beta_1e^{\beta_1b_2}.\end{array}\right.
\end{equation}
From Theorem \ref{veri-thm-asym} we know that if the system $V_i'(\hat b_1)=E_{1,i}{\psi^{(i)}}'(\hat b_i)=1,i=1,2$, (with $E_{1,i}$ as in \eqref{E}) has a solution $(\hat b_1, \hat b_2)$ then it is an equilibrium. Using the definition of $E_{1,i}$ in
\eqref{E}, we rewrite the equation system as
\begin{align}\label{eqsysys}
\frac{K_i}{r}=\frac{{\psi^{(i)}}(\hat b_i)}{{\psi^{(i)}}'(\hat b_i)}-\frac{{\varphi^{(i)}}(\hat b_i)}{{\varphi^{(i)}}'(\hat b_i)}, \enskip i=1,2,
\end{align}
which, in the present case with constant coefficients, is equivalent to
 \begin{equation}\label{eqsyst:part1}
\left\{\begin{array}{l}
\frac{K_1}{r} = \frac{e^{\alpha \hat b_1}- e^{\beta \hat b_1}}{\alpha e^{\alpha_1 \hat b_1} - \beta e^{ \beta \hat b_1}} 
-\frac{F_3 e^{\alpha_1 \hat b_1}- (F_3-1)e^{ \beta_1\hat b_1}}{\alpha_1 F_3 e^{\alpha_1 \hat b_1}- \beta_1(F_3-1)e^{ \beta_1\hat b_1}}\\ 
\frac{K_2}{r} =\frac{e^{\alpha_1\hat b_2}F_1- e^{\beta_1 \hat b_2}F_2}{\alpha_1e^{\alpha_1\hat b_2}F_1- \beta_1e^{\beta_1 \hat b_2}F_2}- \frac{1}{\beta_2}. 
\end{array}\right.
\end{equation} 
Determining expressions for $F_1(\hat b_1)$ and $F_2(\hat b_1)$ (using \eqref{eqsys1}), using these in
the second equation of \eqref{eqsyst:part1}, and solving for $\hat b_2$ yields
\begin{align} \label{asym-ex-calcs-Eq1}
\hat b_2 (\hat b_1)= \frac{\ln\left(\frac{
e^{-\beta_1 \hat b_1}\left(\alpha e^{\alpha \hat b_1}-\beta e^{\beta \hat b_1} -\alpha_1\left(e^{\alpha \hat b_1}-e^{\beta \hat b_1} \right) \right)
}{
e^{-\alpha_1 \hat b_1}\left(\alpha e^{\alpha \hat b_1}-\beta e^{\beta \hat b_1} -\beta_1\left(e^{\alpha \hat b_1}-e^{\beta \hat b_1} \right) \right)
}\frac{1-(K_2/r+1/\beta_2)\beta_1}{1-(K_2/r+1/\beta_2)\alpha_1}\right)}{\alpha_1-\beta_1}.
\end{align} 
The problem of finding a Nash equilibrium $(\hat b_1,\hat b_2)$ is thus effectively transformed to a one-dimensional problem. 
%
Indeed, relying on \eqref{eqsys2} we find 
\begin{align} \label{asdasdas111}
F_3= F_3(\hat b_2)= \frac{-(\beta_1-\beta_2)e^{\beta_1\hat b_2}}{ e^{\alpha_1\hat b_2}(\alpha_1-\beta_2) - e^{\beta_1\hat b_2}(\beta_1-\beta_2)}.
\end{align} 
This expression together with \eqref{asym-ex-calcs-Eq1} induces an explicit mapping $\hat b_1 \mapsto F_3$. Viewing $F_3$ in the first line of \eqref{eqsyst:part1} as this mapping gives us an explicit equation for $\hat b_1$ (which for brevity is not included), which can be easily studied numerically; if it has a solution $\hat b_1$, then $(\hat b_1,\hat b_2)$, with $\hat b_2=\hat b_2(\hat b_1)$ determined according to \eqref{asym-ex-calcs-Eq1}, is an equilibrium; if no solution exists than an equilibrium of either of the other kinds mentioned above can be found using methods similar to that described above. 


Let $\mu=4$, 
$\sigma^2=2$, and  
$r=0.05$, and consider the individual maximal extraction rates $K_1=0.1$ and $K_2=0.2$. Using the program above we determine a Nash equilibrium $(\hat b_1,\hat b_2)=(0.521229,0.704502)$. The corresponding equilibrium value functions are illustrated in Figure \ref{asym:fig}. In Figure \ref{asym:fig2}, the equilibrium thresholds are illustrated when varying one of the maximal extraction rates.   
At least for the chosen parameter values, the effect of a change in $K_2$ on $\hat b_1$ is relatively small.

\begin{figure}[H]
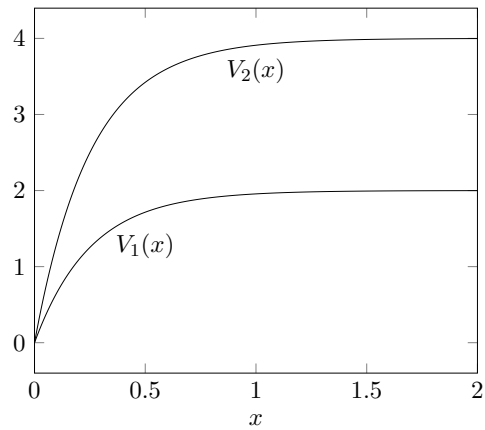
 
\center
 
\caption{The individual equilibrium value functions with $K_1=0.1$ and $K_2=0.2$. 
} \label{asym:fig}
\end{figure}

\begin{figure}[H]
	\center
%
%
\caption{Left: The equilibrium thresholds when $K_2$  is varied and $K_1=0.1$. 
\newline
Right: The equilibrium threshold $\hat b_2(K_2)$ when $K_2$ is varied and $K_1=0.1$ (solid). For comparison we have also included the optimal thresholds in the corresponding one-player game (dashed), which are determined according to Section \ref{sec:constant-coeff-sym}. The straight line is the optimal barrier in the corresponding one-player singular stochastic control problem, cf. Remark \ref{remark-optimal-div}.
} \label{asym:fig2}
\end{figure}

\bibliographystyle{abbrv}
\bibliography{Bibl} 

\end{document}